\documentclass[a4paper, 12pt]{amsart}

\usepackage[T1]{fontenc}
\usepackage[ngerman,english]{babel}
\usepackage{amsmath}
\usepackage{amsthm}
\usepackage{amssymb}
\usepackage{amsfonts}
\usepackage{paralist}
\usepackage{tikz}
\usetikzlibrary{positioning,matrix,arrows,automata}
\usepackage{mathtools}
\usepackage{MnSymbol}

\theoremstyle{plain}
\newtheorem{theorem}{\bf{\textsc{Theorem}}}[section]
\newtheorem*{theoremnonum}{\bf{\textsc{Theorem}}}
\newtheorem*{cornonum}{\bf{\textsc{Corollary}}}
\newtheorem{corollary}[theorem]{\bf{\textsc{Corollary}}}
\newtheorem{lemma}[theorem]{\bf{\textsc{Lemma}}}
\newtheorem{proposition}[theorem]{\bf{\textsc{Proposition}}}
\newtheorem*{conjecture}{\bf\textsc{Conjecture}}
\numberwithin{equation}{section}

\theoremstyle{definition}
\newtheorem{remark}[theorem]{\bf{\textsc{Remark}}}

\newcommand{\BB}{\mathbb{B}}
\newcommand{\CC}{\mathbb{C}}
\newcommand{\DD}{\mathbb{D}}
\newcommand{\LL}{\mathbb{L}}

\newcommand{\NN}{\mathbb{N}}
\newcommand{\ZZ}{\mathbb{Z}}

\newcommand{\sE}{\mathcal{E}}
\newcommand{\sL}{\mathcal{L}}
\newcommand{\sR}{\mathcal{R}}
\newcommand{\sS}{\mathcal S}
\newcommand{\sU}{\mathcal U}
\newcommand{\sV}{\mathcal{V}}
\newcommand{\sW}{\mathcal{W}}

\newcommand{\fraka}{{\mathfrak a}}
\newcommand{\frakg}{{\mathfrak g}}
\newcommand{\frakh}{{\mathfrak h}}
\newcommand{\frakk}{{\mathfrak k}}
\newcommand{\frakn}{{\mathfrak n}}
\newcommand{\frakp}{{\mathfrak p}}
\newcommand{\frakq}{{\mathfrak q}}
\newcommand{\frakz}{{\mathfrak z}}

\DeclareMathOperator{\Poly}{Pol}
\DeclareMathOperator{\End}{End}
\DeclareMathOperator{\Aut}{Aut}
\DeclareMathOperator{\Sym}{Sym}
\DeclareMathOperator{\Op}{Op}
\DeclareMathOperator{\sgn}{sgn}
\DeclareMathOperator{\diag}{diag}
\newcommand{\SymGp}{\mathfrak{S}}
\newcommand{\CR}{{D}}
\newcommand{\CRb}{{\overline D}{}}
\newcommand{\fac}{{\mbox{\tiny fac}}}
\newcommand{\topterm}{{\mbox{\tiny top}}}

\renewcommand{\b}[1]{{\bf #1}}
\newcommand{\TP}[3]{\langle#1#2#3\rangle}
\newcommand{\set}[2]{\{#1\,|\,#2\}}

\title{Higher Laplacians on pseudo-Hermitian symmetric spaces}
\author{Benjamin Schwarz} 

\date{\today}

\address{Benjamin Schwarz, Universit\"{a}t Paderborn,
Fakult\"{a}t f\"{u}r Elektrotechnik, Informatik und Mathematik,
Institut f\"{u}r Mathematik, Warburger Str. 100,
33098 Paderborn, Germany}
\email{bschwarz@math.upb.de}

\begin{document}
\begin{abstract}
Let $X=G/H$ be a symmetric space for a real simple Lie group $G$, equipped with a $G$-invariant complex structure. Then, $X$ is a pseudo-Hermitian manifold, and in this geometric setting, higher Laplacians $L_m$ are defined for each positive integer $m$, which generalize the ordinary Laplace-Beltrami operator. We show that $L_1,L_3,\ldots, L_{2r-1}$ form a set of algebraically independent generators for the algebra $\DD_G(X)$ of $G$-invariant differential operators on $X$, where $r$ denotes the rank of $X$. This confirms a conjecture of Engli\v{s} and Peetre, originally stated for the class of Hermitian symmetric spaces.
\end{abstract}


\maketitle

\section*{Introduction}
Let $X$ be a semisimple symmetric space, i.e., $X=G/H$ for a connected real semisimple Lie group $G$ and closed subgroup $H\subseteq G$ satisfying
\[
	G^\sigma_0\subseteq H\subseteq G^\sigma
\]
for some involution $\sigma$ on $G$, where $G^\sigma$ denotes the fixed points of $\sigma$, and $G^\sigma_0$ is the identity component of $G^\sigma$. One of the most fundamental results concerning harmonic analysis on symmetric spaces states, that the algebra $\DD_G(X)$ of $G$-invariant differential operators on $X$ is a polynomial algebra with $r$ generators, where $r$ is the rank of the symmetric space.

One of these generators can be chosen to be the Laplace--Beltrami operator, whose definition is based on the rich geometric structure underlying the symmetric space. In fact, $X$ can equivalently be defined as a connected pseudo-Riemannian manifold with the property, that the geodesic reflection $s_x$ about any point $x\in X$ extends to a global isometry on $X$, and such that the Ricci tensor associated to the metric is non-degenerate. Then, $G$ can be taken as the displacement group, generated by all products $s_xs_y$ with $x,y\in X$, $H$ is the stabilizer subgroup of some fixed base point $o\in X$, and $\sigma$ is the involution on $G$ given by conjugation with $s_o$. It is known that $G$ coincides with the identity component of the isometry group of $X$, so it is the largest connected semisimple group acting isometrically and effectively on $X$. 

We consider the problem to find a full set of geometrically defined generators for $\DD_G(X)$. In this article, we restrict our attention to semisimple symmetric spaces that admit a metric compatible complex structure, called (semisimple) pseudo-Hermitian symmetric spaces.

For a general pseudo-Hermitian manifold, i.e., a pseudo-Riemannian manifold with metric compatible complex structure, the definition of the Laplace-Beltrami operator naturally generalizes to the definition of a series $L_m$, $m\in\NN$, of higher Laplacians, see Section~\ref{sec:higherlaplacians} for details. Here, $L_m$ is a differential operator of order $2m$, and $L_1$ coincides with the usual Laplace--Beltrami operator. By construction, these operators are invariant under automorphisms (i.e., biholomorphic isometries) of $X$. In the symmetric setting, we thus obtain $G$-invariant operators.

For Hermitian manifolds, higher Laplacians have been introduced by Engli\v{s} and Peetre \cite{EP96}, and for Hermitian symmetric spaces (i.e.\ $H=K$ is a compact subgroup of $G$), they proposed the following conjecture:

\begin{conjecture}[Engli\v{s}--Peetre '96]
	On a simple Hermitian symmetric space $X=G/K$, the higher Laplacians $L_m$, $m\in\NN$,
	generate the whole algebra $\DD_G(X)$ of $G$-invariant differential operators on $X$.
\end{conjecture}

In order to state our main result, recall that the action of $H$ on the tangent space over the base point defines the so called isotropy representation of $X$.

\begin{theoremnonum}
	Let $X=G/H$ be a simple pseudo-Hermitian symmetric space of rank $r$ with irreducible isotropy 
	representation. Then, the higher Laplacians
	\[
		L_1,\ L_3,\ \ldots,\ L_{2r-1}
	\]
	form a set of algebraically independent generators of the algebra $\DD_G(X)$ of $G$-invariant 
	differential operators on $X$. Moreover, for $m$ even, $L_m$ is a polynomial in $L_1,L_3\ldots, 
	L_{m-1}$.
\end{theoremnonum}

In particular, this solves the conjecture of Engli\v{s} and Peetre for simple Hermitian symmetric spaces, since the condition on the isotropy representation is automatic in this case. More precisely, one can show \cite{Sh71}, that the isotropy representation of a simple pseudo-Hermitian symmetric space $X=G/H$ is reducible if and only if $X$ is the 'complexification' of a simple pseudo-Hermitian symmetric space $X'=G'/H'$, i.e., $X=G'_\CC/H'_\CC$. In the reducible case, the higher Laplacians do not generate the full algebra of $G$-invariant differential operators, see Remark~\ref{rmk:reducible}. 

Since the higher Laplacians are invariant not only for the displacement group $G$, but also for the full automorphism group of $X$, we note the following immediate consequence.
\begin{cornonum}
	Let $X=G/H$ be a simple pseudo-Hermitian symmetric space with irreducible isotropy 
	representation. Then, any $G$-invariant differential operator is also $\Aut(X)$-invariant. 
\end{cornonum}
For Hermitian symmetric spaces, this has been noticed before by Engli\v{s} \cite[\S\,5]{En00}.

In its original formulation, the conjecture of Engli\v{s} and Peetre is not restricted to simple Hermitian symmetric spaces, but also includes semisimple Hermitian symmetric spaces. Then, however, it is necessary to replace the displacement group $G$ by the full automorphism group $\Aut(X)$ of $X$, consisting of all biholomorphic isometries of $X$. In \cite{EP96}, Engli\v{s} and Peetre verified this modified conjecture by explicit calculations for the case of the polydisc $X=\BB^N$ with $N=2,3$, where $\BB\subseteq\CC$ is the Poincaré disc. We extend this result to the $N$-fold product $X=X_1\times\cdots\times X_1$ of any rank-1 pseudo-Hermitian symmetric space $X_1$. So far, the full conjecture concerning semisimple Hermitian or pseudo-Hermitian symmetric spaces remains unsolved.\\

\noindent
\textbf{Former results and further comments.}
For simple Hermitian symmetric spaces of rank $1$, the conjecture of Engli\v{s} and Peetre simplifies to the classical result that the Laplace--Beltrami operator generates $\DD_G(X)$. Zhang \cite{Zh00} proved this conjecture for rank $2$, and rank $3$ was settled by Engli\v{s} \cite{En00}. In both approaches, one aims for explicit formulas for the higher Laplacians in terms of algebraically defined generators for $\DD_G(X)$. Such formulas are essentially obtained  by comparing the (full) symbols of the various operators at the base point $o\in X$. However, since the symbol map is not an algebra homomorphism, this approach is computationally demanding, and gets quite complicated for higher order operators. 

One of the main difficulties in dealing with the higher Laplacians is the following: It is known that $\DD_G(X)$ is generated by $r$ differential operators of order $2, 4,\ldots, 2r$. Therefore, one might first assert that $L_1,L_2,\ldots, L_r$ are generators for $\DD_G(X)$. However, a short calculation shows that $L_m$ and $L_1^m$ have the same principal symbol. Similarly, one finds relations for all terms of degree $>m$, and as our analysis shows, it is precisely the term of order $m$ within $L_m$ that determines whether $L_m$ is algebraically independent of $L_1,\ldots, L_{m-1}$.

For the proof of our main theorem, we utilize the realization of $G$-invariant differential operators on $X$ as $H$-invariant elements in the universal algebra $\sU(\frakg)$. Algebraically, the geometric structure of $X$ provides a grading of the complexified Lie algebra, $\frakg_\CC = \frakq_-\oplus\frakh_\CC\oplus\frakq_+$. We show that the higher Laplacian $L_m$ is represented by the element
\begin{align}\label{eq:laplacian}
	L_m \equiv \sum_{\alpha_1,\ldots,\alpha_m=1}^n
			Y_{\alpha_1}\cdots Y_{\alpha_m}X_{\alpha_1}\cdots X_{\alpha_m},
\end{align}
where $(X_{\alpha})_{\alpha=1,\ldots,n}$ and $(Y_{\alpha})_{\alpha=1,\ldots,n}$ are (up to a constant) dual bases of $\frakq_+$ and $\frakq_-$ with respect to the Killing form of $\frakg_\CC$. According to the commutation relations in $\sU(\frakg)$, one can manipulate the representative of $L_m$ in $\sU(\frakg)$, and obtains a variety of different $H$-invariant elements, all having in common that the $H$-invariance is a consequence of the summation of dually paired bases. The main idea is to define and investigate a combinatorial model for such $H$-invariant elements in $\sU(\frakg)$, see Section~\ref{sec:combinatorialmodel} for details.

In the context of Hermitian symmetric spaces with $H=K$, operators similar to the higher Laplacians have originally been introduced by Shimura \cite{S90}: Let $S\frakq_+$ denote the symmetric algebra over $\frakq_+$, considered as a subalgebra of $\sU(\frakg)$. For a simple $K$-module $Z\subseteq S\frakq_+$, Shimura defines the operator
\[
	L_Z = \sum_{\nu=1}^{n_Z} Y_\nu X_\nu,
\]
where $(X_\nu)_{\nu=1,\ldots,n_Z}$ is a basis for $Z\subseteq S\frakq_+$ and $(Y_\nu)_{\nu=1,\ldots,n_Z}$ is the dual basis with respect to the pairing of $S\frakq_+$ and $S\frakq_-$ induced by the Killing form. The restriction to simple $K$-modules is not necessary here, and we note that for $Z=S_m\frakq_+$, the operator $L_Z$ coincides with $L_m$ up to a constant multiple. It follows that $L_m$ is a sum of Shimura operators, as Zhang already noted \cite{Zh00}. Shimura determines a set $Z_1,\ldots, Z_r$ of simple $K$-modules such that $L_{Z_1},\ldots,L_{Z_r}$ are algebraically independent generators of $\DD_G(X)$. We note however, that these operators are not geometrically defined operators as it is the case for the higher Laplacians.\\

\noindent
\textbf{Organization.} In the first section, we introduce Cauchy--Riemann operators and higher Laplacians on general pseudo-Hermitian manifolds. This is a slight generalization of the definition given by Engli\v{s} and Peetre \cite{EP96}. In Section~\ref{sec:HLonpHSS}, we recall the basic structure theory for pseudo-Hermitian symmetric spaces, and briefly discuss the connection between $G$-invariant differential operators, the universal enveloping algebra and the Harish-Chandra isomorphism. Our main Theorem~\ref{thm:mainthm} is stated. We define a special set of generators for $\DD_G(X)$ which is needed to identify certain terms in the higher Laplacians. Finally, we identify the higher Laplacian $L_m$ with an element of the universal enveloping algebra as stated in \eqref{eq:laplacian}. In the third section we develop the combinatorial model for a certain class of $G$-invariant differential operators, which is then used to prove our main theorem. In Section~\ref{sec:nonIrreducible} we briefly consider semisimple pseudo-Hermitian spaces. We show that the arguments in the proof of our main theorem still apply in the case of the $N$-fold product of rank-1 pseudo-Hermitian symmetric spaces.\\

\noindent
\textbf{Acknowledgements.} I like to thank Jean-Stefan Koskivirta for giving a counter example on a crucial lemma of the first version of this article. This helped to clarify the arguments, which fortunately turned out still to be applicable, eventually in the present form of this article. I am also thankful to Job Kuit and Maarten van Pruijssen for various discussions about semisimple symmetric spaces.

\section{Higher Laplacians on pseudo-Hermitian manifolds}\label{sec:higherlaplacians}
We define higher Laplacian operators in the general setting of pseudo-Hermitian manifolds following the exposition in \cite{EP96}, where the Hermitian case is discussed.

Let $(X,g,J)$ be a pseudo-Hermitian manifold, i.e., $X$ is a smooth manifold equipped with a pseudo-Riemannian metric $g$ and a metric compatible complex structure $J$. Let $T_\CC=T^{1,0}\oplus T^{0,1}$ denote the splitting of the complexified tangent bundle of $X$ into the holomorphic and the anti-holomorphic tangent bundle, and let $h\colon T^{1,0}\times T^{0,1}\to\CC$ be the restriction of the bilinearly extended complexified metric $g_\CC$. Then, $h$ is non-degenerate and hence defines a $\CC$-linear pairing of $T_x^{1,0}$ and $T_x^{0,1}$ for each $x\in X$. Let $E\to X$ be a holomorphic vector bundle on $X$. The \emph{covariant Cauchy--Riemann operator} $\CRb_E$ of $E$ is defined by the composition
\begin{align*}
  \begin{tikzpicture}[baseline=-6mm]
    \node (C1) at (0,0) {$C^\infty(X,E)$};
    \node[right=of C1] (C2) {$C^\infty(X,E\otimes (T^{0,1})^*)$};
    \node[right=of C2] (C3) {$C^\infty(X,E\otimes T^{1,0})$,};
    \draw[->] (C1)--(C2) node [midway,above] {\small$\overline\partial$};
    \draw[->] (C2)--(C3) node [midway,above] {\small$h_*$};
		\draw[->,bend right=20] (C1) to node [below] {$\CRb_E$} (C3);
	\end{tikzpicture}
\end{align*}
where $h_*$ denotes is the canonical isomorphism induced by the pairing $h$. Since $E\otimes T^{1,0}$ is again a holomorphic vector bundle, iterates of the Cauchy--Riemann operator are defined in the obvious way. By abuse of notation, we simply write
\[
	\CRb^m :=\CRb\circ\cdots\circ\CRb\colon C^\infty(X,E)\to
	C^\infty(X,E\otimes (T^{1,0})^{\otimes m}).
\]
As a remark, we note that if $(X,g,J)$ is (pseudo-)Kählerian, i.e., the 2-form $g(J\cdot,\cdot)$ is closed, then one can show that $\CRb^mf$ is actually a section in $E\otimes\Sym_m$, where $\Sym_m\subseteq(T^{1,0})^{\otimes m}$ denotes the subbundle of symmetric tensors. This observation is originally due to Shimura \cite[Lemma~2.0]{S86}, and is used extensively in the theory of nearly holomorphic sections \cite{Sc14,Sc13a,S87,Zh02}. For the present article, this additional structure is not needed, even though we note that pseudo-Hermitian symmetric spaces are automatically (pseudo-)Kählerian.

In order to define higher Laplacians, we fix a pseudo-Hermitian structure on $E$, i.e., a smoothly varying non-degenerate (not necessarily positive) Hermitian form $\langle\,|\,\rangle_x$ on each fiber $E_x$, $x\in X$. For the holomorphic tangent bundle $T^{1,0}$ such a structure is given by $h(X_1,\overline X_2)$, where $\overline X_2$ denotes complex conjugation on $T_\CC$ with respect to $J$. For compactly supported sections $f_1,f_2$ in $E$, we define
\[
	\langle f_1|f_2\rangle_E:=\int_X\langle f_1(x)|f_2(x)\rangle_x\,\omega_X(x),
\]
where $\omega_X$ denotes the volume form given by the pseudo-Riemannian metric $g$ on $X$. The formal adjoint $\CR^m:=(\CRb^m)^*$ of $\CRb^m$ is defined by the relation
\begin{align}\label{eq:formaladjoint}
	\langle\CRb^mf|g\rangle_{E\otimes(T^{1,0})^{\otimes m}} = \langle f|\CR^m g\rangle_E
\end{align}
for all compactly supported sections $f$ in $E$ and $g$ in $E\otimes(T^{1,0})^{\otimes m}$. Then, 
\[
	\CR^m\colon C^\infty(X,E\otimes(T^{1,0})^{\otimes m})\to C^\infty(X,E)
\]
is again a differential operator of order $m$, and the \emph{higher Laplacian operator} $L_m$ is the differential operator of order $2m$ defined by
\[
	L_m:=(-1)^m\CR^m\CRb^m\colon C^\infty(X,E)\to C^\infty(X,E).
\]
In particular, for the trivial line bundle $E=X\times\CC$, we obtain operators $L_1,L_2,\ldots$ acting on functions on $X$. In holomorphic coordinates $z=(z^1,\ldots, z^n)$ on an open subset of $X$, one easily derives that the first Laplacian $L_1$ is given by
\[
	 L_1f = -\CR\CRb f = \sum_{i,j=1}^n\frac{1}{|h|}\frac{\partial}{\partial z^i}
	 	\left(|h|h^{ji}\frac{\partial f}{\partial\bar z^j}\right),
\]
where $|h|$ is the determinant of the matrix $h_{ij}:=h(\frac{\partial}{\partial z^i},\frac{\partial}{\partial\bar z^j})$, and $h^{ji}$ denote the entries of the inverse matrix of $h_{ij}$. This is the well-known local formula for the ordinary Laplace--Beltrami operator of $X$. 

It is straightforward to see that the Cauchy--Riemann operator $\CRb^m$ and its adjoint $\CR^m$ commute with the action of biholomorphic Hermitian bundle isomorphisms of $E$ which act isometrically on the base $X$. Therefore, each higher Laplacian $L_m$ is invariant under such isomorphisms. 

\section{Higher Laplacians on simple pseudo-Hermitian symmetric spaces}\label{sec:HLonpHSS}
\subsection{Pseudo-Hermitian symmetric spaces}
Let $X$ be a semisimple pseudo-Hermitian symmetric space as defined (in geometric terms) in the introduction. Let $G$ be the displacement group of $X$, which is a connected semisimple Lie group. In fact, $G$ coincides with the connected identity component of the automorphism group $\Aut(X)$ of $X$, which consists of biholomorphic isometries. Let $o\in X$ be a fixed base point with symmetry $s_o$, and let $H\subseteq G$ be the stabilizer subgroup of $o$ in $G$. Let $\sigma$ denote the involution of $G$ given by conjugation with $s_o$. We recall some standard facts about pseudo-Hermitian symmetric spaces, and refer to \cite{Sh71} for detailed information.

Any semisimple pseudo-Hermitian symmetric space is simply connected, and hence decomposes according to the deRham--Wu decomposition \cite{Wu64} into a product of simple pseudo-Hermitian symmetric spaces. For the following, we assume that $X$ is simple, which corresponds to the assumption that the displacement group $G$ is simple. For this, we note that a simple complex Lie group $G$, considered as a pseudo-Riemannian symmetric space via the standard construction $G\cong G\times G/\diag(G)$, is not a pseudo-Hermitian manifold, since the complex structure of $G$ is not compatible with the pseudo-Riemannian metric.

Let $\frakg$, $\frakh$ denote the Lie algebras of $G$ and $H$, and let $\sigma$ also denote the involution induced by the adjoint action of $s_o$ on $\frakg$. Let
\begin{align}\label{eq:sigmadecomp}
	\frakg = \frakh\oplus\frakq
\end{align}
be the decomposition of $\frakg$ into the $+1$ and $-1$ eigenspace of $\sigma$. Complexification of a Lie algebra is denoted by a corresponding index, e.g.\ $\frakg_\CC:=\frakg\oplus i\frakg$. For $X=X_1+iX_2\in\frakg$, let $\overline X:=X_1-iX_2$ denote complex conjugation on $\frakg_\CC$. Geometrically, $\frakq$ corresponds to the real tangent space at $o\in X$, $\frakq_\CC$ corresponds to the complex tangent space, and the complex structure $J_o$ on $T_oX$ induces the decomposition
\begin{align}\label{eq:qdecomp}
	\frakq_\CC = \frakq_-\oplus\frakq_+,
\end{align}
whose parts can be identified with the holomorphic and the anti-holo\-morphic tangent space at $o\in X$. We note that $\overline{\frakq_+}=\frakq_-$. The metric on $\frakq\cong T_oX$ defines a non-degenerate symmetric complex bilinear form
\begin{align}\label{eq:metric}
	(\,|\,)\colon\frakq_\CC\times\frakq_\CC\to\CC.
\end{align}
Since $X$ is simple, this form coincides up to a constant multiple with the Killing form of $\frakg_\CC$, restricted to $\frakq_\CC$. Therefore, \eqref{eq:metric} extends to a bilinear form on $\frakg_\CC$, and this extension is associative, i.e., $([X,Y]|Z) = (X|[Y,Z])$. The restriction of \eqref{eq:metric} to $\frakq_+\times\frakq_-$ correlates to the pairing $h$ of the holomorphic and the anti-holomorphic tangent bundle at the base point $o\in X$.\\
The existence of a metric compatible complex structure on $X$ corresponds to the fact that $\frakh$ has non-trivial center acting with pure imaginary eigenvalues on $\frakq$. Moreover, the decomposition \eqref{eq:qdecomp} is also induced by the adjoint action of a unique central element $Z_0\in i\frakz(\frakh)$, and it follows that $\frakg$ admits the grading 
\[
	\frakg = \frakq_-\oplus\frakh_\CC\oplus\frakq_+,\quad\frakq_\pm=\set{X\in\frakg}{[Z_0,X]=\pm X}.
\]
In particular, $\frakq_\pm$ are abelian subalgebras. The adjoint action of $H$ on $\frakq$ is called the isotropy representation of $X$. Even though $X$ is simple, this representation can be reducible. More precisely, the isotropy representation of $X$ is irreducible if and only if the complex Lie algebra $\frakg_\CC$ is simple. In this case, the 'complexified' pseudo-Hermitian symmetric space $X_\CC=G_\CC/H_\CC$ is again a simple pseudo-Hermitian symmetric space, but with reducible isotropy representation.

There is a non-compact Hermitian symmetric space $X^r$ associated to $X$, called the Hermitian form of $X$, by the following construction: Let $\theta$ be a Cartan involution of $\frakg$ which commutes with $\sigma$. Then, the Cartan decomposition $\frakg =\frakk\oplus\frakp$ is compatible with \eqref{eq:sigmadecomp}, and we obtain
\begin{align*}
	\frakg
	=(\frakk\cap\frakh)\oplus(\frakk\cap\frakq)\oplus(\frakp\cap\frakh)\oplus(\frakp\cap\frakq).
\end{align*}
The Hermitian form of $X$ is defined by $X^r:=G^d/K^d$ with $G^d$ and $K^d$ determined by their Lie algebras $\frakg^d$ and $\frakk^d$,
\begin{align*}
	\frakg^d := \frakk^d\oplus\frakp^d\quad\text{with}\quad\left\{
	\begin{aligned}
		\frakk^d&:=(\frakk\cap\frakh)\oplus i(\frakp\cap\frakh),\\
		\frakp^d&:=(\frakp\cap\frakq)\oplus i(\frakk\cap\frakq).
	\end{aligned}\right.
\end{align*}
Indeed, if $\sigma_\CC$ denotes the complex linear extension of $\sigma$ to $\frakg_\CC$, then $\theta^d:=\sigma_\CC|_{\frakg^d}$ is a Cartan involution for $\frakg^d$ with fixed point set $\frakk^d$. Moreover, one can show that $iZ_0$ is fixed by $\theta$ (see e.g.\ \cite[Proposition~10]{BN09}), hence $iZ_0\in\frakk\cap\frakh$ is a non-trivial central element of $\frakk^d$. This shows, that the Riemannian symmetric space $X^r=G^d/K^d$ associated to the pair $(\frakg^d,\frakk^d)$ is indeed Hermitian. We also note that 
\[
	\frakk^d = \frakg^d\cap\frakh_\CC,\ \frakp^d = \frakg^d\cap\frakq_\CC
	\quad\text{and}\quad
	\frakp_+=\frakq_+,\ \frakp_-=\frakq_-,
\]
where $\frakp_\pm$ is the $\pm 1$-eigenspace of $\frakp^d_\CC$ corresponding to the complex structure (induced by the adjoint action of $Z_0$) on $\frakp^d$. A straightforward calculation also shows that the isotropy representation of $X$ is irreducible if and only if the Hermitian form $X^r$ is simple. In the reducible case, $X^r$ splits into the product $X^r=X^r_0\times X^r_0$ of two copies of a non-compact simple Hermitian symmetric space $X_0^r$.

Recall that a Cartan subspace of $\frakq$ is a maximal abelian subalgebra $\fraka\subseteq\frakq$ consisting of semi-simple elements. Any Cartan subspace has the same dimension, $r=\dim\fraka$, which is called the \emph{rank} of $X$. We will need a particular choice of a basis for a Cartan subspace.

\begin{lemma}\label{lem:CartanBasis}
	Let $\fraka\subseteq\frakq$ be a Cartan subspace. There exists a basis $A_1,\ldots,A_r$ of 
	$\fraka_\CC$ satisfying the following properties: Let $E_j\in\frakq_+$ and 
	$F_j\in\frakq_-$ be defined by $A_j=E_j+F_j$, and $H_j:=[E_j,F_j]$. Then,
	\begin{align}\label{eq:commutators}
		[E_j,F_j]=:H_j,\quad [H_j,E_j]=E_j,\quad [H_j,F_j]=-F_j,
	\end{align}
	and $[E_j,F_k]=0$ for $j\neq k$. Moreover, $(E_j|F_j)$ is non-zero and independent of
	$1\leq j\leq r$. Set
	\begin{align}\label{eq:structureconstant}
		c_0:=(E_j|F_j).
	\end{align}
	Then, $(A_j|A_k)=2c_0\,\delta_{jk}$ for $1\leq j,k\leq r$.
\end{lemma}
\begin{proof}
We reduce this statement to a well-known result on non-compact Hermitian symmetric spaces by use of the Hermitian form $X^r$ of $X$ defined above. Recall that any Cartan subspace $\fraka$ of $\frakq$ is $H$-conjugate to a $\theta$-invariant Cartan subspace. We thus may assume that $\fraka$ is $\theta$-invariant. Then, $\fraka^d=i(\fraka\cap\frakk)\oplus(\fraka\cap\frakp)$ is a Cartan subspace of $\frakp^d$. For non-compact Hermitian symmetric spaces as $X^r$, Harish-Chandra introduced a Cartan subspace in $\frakp^d$ with appropriate basis by his construction of strongly orthogonal roots, see \cite[II.\S\,6]{HC55}. Moreover, due to Moore \cite{Mo64}, all strongly orthogonal roots have the same length. Since any two Cartan subspaces in $\frakp^d$ are conjugate by an element of $K^d$, this provides a real basis of $\fraka^d$ satisfying all properties we demand. Since $\fraka^d_\CC=\fraka_\CC$, this proves our statement.
\end{proof}

\begin{remark}
	For a Cartan subspace $\fraka\subseteq\frakq$, the basis $A_1,\ldots,A_r$ of $\fraka_\CC$ 
	given by Lemma~\ref{lem:CartanBasis} can in general not be chosen as real a basis for $\fraka$, 
	since the relation $(A_j|A_k)=2c_0\,\delta_{jk}$ would imply that the metric $g$ on $X$ is 
	positive (respectively negative) definite.
\end{remark}

\subsection{Invariant differential operators}
Let $\DD_G(X)$ denote the algebra of $G$-invariant operators on $X$. Examples of such operators are given by the higher Laplacians $L_m$, $m\geq0$, defined in the first section. A fundamental result on symmetric spaces mainly due to Harish-Chandra asserts that the algebra $\DD_G(X)$ is commutative and finitely generated by $r$ algebraically independent elements, where $r$ is the rank of $X$. The main goal of this article is to prove the following result.

\begin{theorem}\label{thm:mainthm}
	Let $X=G/H$ be a simple pseudo-Hermitian symmetric space of rank $r$ with irreducible isotropy 
	representation. Then, the higher Laplacians
	\[
		L_1,\ L_3,\ldots,\ L_{2r-1}
	\]
	form a set of algebraically independent generators for $\DD_G(X)$. Moreover, for $m$ even, $L_m$ 
	is a polynomial in $L_1,L_3,\ldots,L_{m-1}$.
\end{theorem}

Recall that Harish-Chandra's result is established essentially in two steps. One first identifies $H$-invariant elements of the universal enveloping algebra $\sU(\frakg)$ of $\frakg_\CC$ with $G$-invariant differential operators on $X$ in the following way: Describe smooth functions on $X$ as right $H$-invariant smooth functions on $G$,
\[
	C^\infty(X)\cong\set{f\colon G\to\CC\text{ smooth}}{f(gh)=f(g)\text{ for all $g\in G,h\in H$}},
\]
and extend the right action of $\frakg$ on $f\in C^\infty(G)$ to a right action of $\sU(\frakg)$ on $C^\infty(G)$. Then, the action of $H$-invariant elements $Y\in\sU(\frakg)^H$ restricts to an action on $C^\infty(X)$. The result is a surjective morphism of algebras,
\begin{align*}
	\sR\colon\sU(\frakg)^H\to\DD_G(X),\ Y\mapsto\sR_Y,
\end{align*}
and it turns out that the kernel is $\sU(\frakg)\frakh_\CC\cap\sU(\frakg)^H$, so
\begin{align}\label{eq:firstisomorphism}
	\sU(\frakg)^H/(\sU(\frakg)\frakh_\CC\cap\sU(\frakg)^H)\cong\DD_G(X).
\end{align}
This identification of algebras holds more generally for any reductive homogeneous space, see e.g.\ \cite{He84}.\\
The second step is specific to symmetric spaces and identifies the left hand side of \eqref{eq:firstisomorphism} with the algebra of $W$-invariant elements in the symmetric algebra $S\fraka_\CC$, where $\fraka\subseteq\frakq$ is a Cartan subspace, and $W$ is the Weyl group attached to the (restricted) root system $\Phi(\frakg,\fraka_\CC)$. More precisely, choose a positive system in $\Phi(\frakg,\fraka_\CC)$ and let $\frakg=\frakn_\CC\oplus\fraka_\CC\oplus\frakh_\CC$ be the corresponding complex Iwasawa decomposition. This induces a vector space decomposition of the enveloping algebra, 
\[
	\sU(\frakg)=(\frakn_\CC\sU(\frakg)+\sU(\frakg)\frakh_\CC)\oplus S\fraka_\CC,
\]
which defines the projection
\begin{align}\label{eq:aprojection}
	\sU(\frakg)\to S\fraka_\CC,\ Y\mapsto Y_\fraka.
\end{align}
Considering $Y_\fraka$ as a differential operator on $\fraka$, the main result in this context states that 
\begin{align}\label{eq:secondisomorphism}
	\sU(\frakg)^H\to (S\fraka_\CC)^W,\ Y\mapsto e^{-\rho}\, Y_\fraka\circ e^{\rho}
\end{align}
is a surjective morphism of algebras with kernel $\sU(\frakg)\frakh_\CC\cap\sU(\frakg)^H$, where $\rho\in\fraka^*_\CC$ is half the trace of the adjoint action on $\frakn_\CC$. In combination, \eqref{eq:firstisomorphism} and \eqref{eq:secondisomorphism} lead to the famous Harish-Chandra isomorphism
\begin{align}\label{eq:CHiso}
	\gamma\colon\DD_G(X)\to(S\fraka_\CC)^W.
\end{align}
Finally, the work of Chevalley on finite reflection groups shows that the algebra of $W$-invariants in $S\fraka_\CC$ is generated by $r=\dim\fraka$ algebraically independent elements.

In practice, however, it is notoriously hard to determine $\gamma(D)$ for an explicitly given operator $D\in\DD_G(X)$, say except for the Laplace-Beltrami operator $L_1$. In order to prove Theorem~\ref{thm:mainthm}, we analyze the higher Laplacians within a particular class of $H$-invariant elements in $\sU(\frakg)$. This class of operators is defined by a combinatorial model, see Section~\ref{sec:combinatorialmodel}. Our analysis shows that the higher Laplacian $L_m$ is a polynomial of $G$-invariant operators of order $\leq m$. Finally, we compare the operators of order $m$ within this polynomial via the Harish-Chandra isomorphism with a given set of generators for $\DD_G(X)$. This will eventually prove Theorem~\ref{thm:mainthm}.

\begin{remark}\label{rmk:reducible}
  We note that Theorem~\ref{thm:mainthm} fails for simple pseudo-Hermitian symmetric spaces with 
  reducible isotropy representation: Consider the Hermitian dual $X^r$ of $X$. It is a consequence 
  of \eqref{eq:firstisomorphism}, that the algebras $\DD_G(X)$ and $\DD_{G^d}(X^r)$ are 
  naturally isomorphic. If the isotropy representation is reducible, then $X^r$ splits 
  into the product $X^r=X_0^r\times X_0^r$ of two copies of a simple non-compact Hermitian 
  symmetric spaces $X_0^r=G_0/H_0$, and $G^d=G_0\times G_0$. In this case, there exist 
  $G^d$-invariant differential operators on $X^r$ that are not invariant under the full 
  automorphism group of $X^r$, which in particular contains the permutation of the two copies of 
  $X_0^r$ in $X^r$. Since the higher Laplacians are invariant under arbitrary automorphisms of 
  $X^r$, it follows that the higher Laplacians cannot generate $\DD_{G^d}(X^r)$ and hence 
  also cannot generate $\DD_G(X)$.
\end{remark}

\subsection{Generators}\label{subsec:generators}
We fix a particular set of generators for $\DD_G(X)$ by means of the Harish-Chandra isomorphism $\gamma$. Let $A_1,\ldots,A_r$ be the basis of $\fraka_\CC$ given by Lemma~\ref{lem:CartanBasis}, and let $A_1',\ldots,A_r'$ denote the basis dual to $A_1,\ldots,A_r$ with respect to \eqref{eq:metric}. In fact, $A_j'=\tfrac{1}{2c_0}\,A_j$. Instead of considering the symmetric algebra $S\fraka_\CC$, we prefer to work in the dual setting, i.e., we identify $\fraka_\CC$ with its dual $\fraka_\CC^*$ with respect to \eqref{eq:metric}, which also induces an algebra isomorphism between the symmetric algebra $S\fraka_\CC$ and the polynomial algebra $\Poly(\fraka_\CC)$. For $k\in\NN$, set
\begin{align}\label{eq:polynomials}
	p_k(\zeta) := \zeta_1^{2k}+\cdots+\zeta_r^{2k},\quad\zeta\in\fraka_\CC,
\end{align}
where $\zeta_j:=(\zeta|A_j')$, i.e., $\zeta=\sum\zeta_jA_j$. As before, let $W$ be the Weyl group of the restricted root system $\Phi(\frakg_\CC,\fraka_\CC)$. 

\begin{lemma}\label{lem:generators}
	If $X$ is simple with irreducible isotropy representation, then the polynomials
	$p_1,\ldots, p_r$ are $W$-invariant and form a set of algebraically independent generators 
	for $\Poly(\fraka_\CC)^W$.
\end{lemma}
\begin{proof}
As in the proof of Lemma~\ref{lem:CartanBasis}, consider the Hermitian form $X^r$ of $X$, which yields the same restricted root system. Irreducibiliy of the isotropy representation now implies that $G^d$ is simple. Due to Moore \cite{Mo64}, it then follows that $\Phi(\frakg_\CC,\fraka_\CC)$ is of type $C_r$ or $BC_r$. In either case, the Weyl group $W$ consists of signed permutations of the basis vectors $A_j$ in $\fraka_\CC$. Therefore, $W$-invariant polynomials are symmetric polynomials in $\zeta_1^2,\ldots,\zeta_r^2$, and by Newton's identities it is known that $p_1,\ldots,p_r$ from a set of algebraically independent generators for $\Poly(\fraka_\CC)^W$.
\end{proof}

By means of the Harish-Chandra isomorphism, we define
\[
	D_k:=\gamma^{-1}(p_k)\in\DD_G(X), k\in\NN.
\]
For simple $X$ with irreducible isotropy representation, it follows from Lemma~\ref{lem:generators} that $D_1,\ldots,D_r\in\DD_G(X)$ form a set of algebraically independent generators for $\DD_G(X)$. For later use, we note the following observation, which is a consequence of the homogeneity of the generators $p_1,\ldots, p_r$.

\begin{lemma}\label{lem:lowdegreeoperatros}
	Let $X$ be simple with irreducible isotropy representation, and let $D\in\DD_G(X)$. 
	Then, $D$ has even order $2k$, $k\in\NN$. Moreover, if $k\leq r$, then there exists a unique 
	polynomial $P$ in $k$ variables, such that $D=P(D_1,\ldots,D_k)$.
\end{lemma}
\begin{proof}
Due to the Harish-Chandra isomorphism, this statement translates to a statement on $W$-invariant polynomials on $\fraka_\CC$, since $\gamma$ relates differential operators of order $d$ with polynomials of degree $d$. Therefore, since the generators $p_1,\ldots, p_r$ have even degrees, a $G$-invariant operator necessarily has even order. Now let $p$ be a $W$-invariant polynomial. Since $p_1,\ldots, p_r$ are algebraically independent, there exists a unique polynomial $P(t_1,\ldots,t_r)$ in $r$ variables, such that $p=P(p_1,\ldots, p_r)$. It remains to show that for $k<r$, $P$ is independent of $t_{k+1},\ldots,t_r$. We may assume that $p$ is homogeneous, since homogeneous components of $W$-invariant polynomials are again $W$-invariant. Since $p_k$ is homogeneous of degree $2k$, it follows that
\begin{align}\label{eq:homogeneous}
	\lambda^{2k}\,p=P(\lambda^2p_1,\lambda^4p_2,\cdots,\lambda^{2r}p_r)\quad
	\text{for all}\quad\lambda\in\CC.
\end{align}
Considering the expansion $P = \sum_{\b i} a_\b i t^\b i$ with multi-indices $\b i=(i_1,\ldots, i_r)$, and comparing the degrees of $\lambda$ in \eqref{eq:homogeneous}, it follows that $a_\b i=0$ if $i_\ell>0$ for some $\ell>k$. Therefore, $P$ just depends on $t_1,\ldots, t_k$.
\end{proof}

\subsection{Higher Laplacians}
We identify the higher Laplacians with $H$-invariant elements in the universal enveloping algebra $\sU(\frakg)$ along the isomorphism given in \eqref{eq:firstisomorphism}. Recall that the complexified metric \eqref{eq:metric} defines a bilinear pairing of $\frakq_+$ and $\frakq_-$. We fix a basis $(X_\alpha)_{\alpha=1,\ldots,n}$ of $\frakq_+$, and let $(Y_\alpha)_{\alpha=1,\ldots,n}$ denote the basis of $\frakq_-$ dual to $(X_\alpha)_{\alpha=1,\ldots,n}$.

\begin{proposition}
	The algebra isomorphism \eqref{eq:firstisomorphism} identifies the higher Laplacian $L_m$ with
	\[
		L_m=\sum_{\alpha_1,\ldots,\alpha_m=1}^n
			Y_{\alpha_1}\cdots Y_{\alpha_m}X_{\alpha_1}\cdots X_{\alpha_m}
			\mod \sU(\frakg)\frakh_\CC\cap\sU(\frakg)^H.
	\]
\end{proposition}
\begin{proof}
We consider the action of $L_m$ on smooth functions $f$, represented as right $H$-invariant functions on $G$. We first recall some basic facts about holomorphic vector bundles on $X$. A smooth $G$-homogeneous vector bundle $E$ can be realized as a fibered product $E=G\times^H E_o$, where the canonical fiber $E_o$ is an $H$-module. This also defines a representation of $\frakh_\CC$ on $E_o$, denoted by $\tau\colon\frakh_\CC\to\End(E_o)$. According to \cite{TW70}, a holomorphic structure on $E$ corresponds to an extension of $\tau$ to a representation of $\frakh_\CC\oplus\frakq_+\subseteq\frakg_\CC$. Moreover, a smooth section in $E$ is represented by a smooth map $f\colon G\to E_o$ satisfying $f(gh)=h^{-1}.f(g)$ for all $g\in G$, $h\in H$, and $f$ is holomorphic if and only if $\sR_Xf(g)+X.f(g)=0$ for all $g\in G$, $X\in\frakq_-$. It follows, that the del-bar operator $\bar\partial$ on $E$ is given by
\[
	\bar\partial f\colon G\to E\otimes(\frakq_+)^*,\
	\bar\partial f(g)(X)=\sR_X f(g)+\tau(X)f(g),\quad X\in\frakq_+.
\]
In case of the holomorphic and the anti-holomorphic tangent bundles,
\[
	T^{1,0} = G\times^H\frakq_-,\quad T^{0,1} = G\times^H\frakq_+,
\]
the action of $H$ on the canonical fibers $\frakq_+$ and $\frakq_-$ is given by the adjoint representation, and $\frakq_+$ acts trivially in both cases. The same holds for the $m$'th tensor product $(T^{1,0})^{\otimes m}$ with canonical fiber $\frakq_-^{\otimes m}$. Therefore, the del-bar operator on a section in $(T^{1,0})^{\otimes m}$ simplifies to $\bar\partial f(g)(X)=\sR_Xf(g)$, and it follows that the action of the Cauchy--Riemann operator $\CRb$ on $f$ is given by
\[
	\CRb f(g)=\sum_{\alpha=1}^n\sR_{X_\alpha}f(g)\otimes Y_\alpha.
\]
As a consequence, the $m$'th iterate of $\CRb$ on a function $f\in C^\infty(X)$ reads
\begin{align}\label{eq:CRoperatorformula}
	\CRb^mf(g)
		= \sum_{\alpha_1,\ldots,\alpha_m=1}^n\sR_{X_{\alpha_1}}\cdots\sR_{X_{\alpha_m}}f(g)
			\,Y_{\alpha_m}\otimes\cdots\otimes Y_{\alpha_1}.
\end{align}
It remains to determine the formal adjoint of $\CRb^m$ defined by \eqref{eq:formaladjoint}. Recall from \cite[\S\,5.5]{Wa73} the following result: Let $E=G\times^H E_o$ and $F=G\times^H F_o$ be $G$-homogeneous vector bundles on $X$, and fix $H$-invariant non-degenerate Hermitian forms on $E_o$ and $F_o$, which define non-degenerate Hermitian forms on $E$ and $F$. The natural action of an element $D\in (\End(E_o,F_o)\otimes\sU(\frakg))^H$ on $E_o$-valued smooth maps on $G$ defines a differential operator from  $C^\infty(X,E)$ to $C^\infty(X,F)$. Then, the adjoint of $D$ with respect to the Hermitian forms on $E$ and $F$ is given by $D^*=(\eta\otimes\delta) D$, where $\eta(T):=T^*$ is the adjoint of $T\in\End(E_o,F_o)$ with respect to the Hermitian forms on $E_o$ and $F_o$, and $\delta$ is the antiautomorphism of $\sU(\frakg)$ which extends $\delta(X)=-\overline X$ on $\frakg_\CC$.\\
In our situation, $E_o=\CC$ with standard Hermitian form, and $F_o=\frakq_-^{\otimes m}$ with Hermitian form defined by
\[
	\langle Y_1\otimes\cdots\otimes Y_m|Y_1'\otimes\cdots\otimes Y_m'\rangle
	:=(Y_1|\overline{Y_1'})\cdots(Y_m|\overline{Y_m'}).
\]
It follows, that the adjoint $\CR^m$ of $\CRb^m$ acts on $f=\sum f_{\beta_1\ldots\beta_m}Y_{\beta_1}\otimes\cdots\otimes Y_{\beta_m}$ by the formula
\begin{align*}
	\CR^mf(g) &= (-1)^m\sum_{\substack{\alpha_1,\ldots,\alpha_m\\\beta_1,\ldots,\beta_m}}
		\sR_{\overline X_{\alpha_m}}\cdots\sR_{\overline X_{\alpha_1}}f_{\beta_1\ldots\beta_m}(g)
		(Y_{\beta_1}|\overline{Y_{\alpha_1}})\cdots(Y_{\beta_m}|\overline{Y_{\alpha_m}})\\
		&= (-1)^m\sum_{\beta_1,\ldots,\beta_m}\sR_{Y_{\beta_m}}\cdots\sR_{Y_{\beta_1}} 
				f_{\beta_1\cdots\beta_m}.
\end{align*}
Here, we used the standard formula $Y=\sum_{\alpha}(Y|\overline Y_{\alpha})\overline X_\alpha$ valid for any $Y\in\frakq_+$. Since $\frakq_-$ is abelian, we may freely change the order of the operators $\sR_{Y_{\beta_i}}$. In combination with \eqref{eq:CRoperatorformula}, this proves the formula for the higher Laplacian $L_m$.
\end{proof}

\section{A combinatorial model}\label{sec:combinatorialmodel}
As we have seen in the last section, the higher Laplacian $L_m$ is represented within the universal enveloping algebra $\sU(\frakg)$ by the element
\[
	\LL_m	= \sum_{\alpha_1,\ldots,\alpha_m=1}^n
					Y_{\alpha_1}\cdots Y_{\alpha_m}X_{\alpha_1}\cdots X_{\alpha_m},
\]
where $(X_\alpha)_{\alpha}$ is a fixed basis of $\frakq_+$, and $(Y_{\alpha})_\alpha$ is the dual basis of $\frakq_-$ with respect to \eqref{eq:metric}. In this formula, $H$-invariance of $\LL_m$ is ensured by taking sums over pairs of dual vectors $X_{\alpha_j}$ and $Y_{\alpha_j}$. Any permutation of the symbols involved in $\LL_m$ also yields to an $H$-invariant element in $\sU(\frakg)$, and hence to a $G$-invariant differential operator on $X$, which in general essentially differs from $L_m$. More generally, we could consider commutators of various of the symbols in $\LL_m$ and still obtain an $H$-invariant element. It is our goal in this section to study this class of $H$-invariant elements, and the corresponding $G$-invariant differential operators. We first formally define the class of elements under consideration.

\subsection{Formal language}
We introduce a formal language consisting of symbols, letters and words, which then determines a class of $G$-invariant operators containing the higher Laplacian $L_m$. Throughout this section, we fix $m\in\NN$. Let $\sS$ be the set of \emph{symbols}, defined by
\[
	\sS:=\sS_+\cup\sS_-\quad\text{with}\quad
	\sS_+:=\{x_1,\ldots,x_m\},\quad
	\sS_-:=\{y_1,\ldots,y_m\}.
\]
Based on these symbols, we recursively define two types of \emph{letters}, $\sL=\sL_+\cup\sL_-$, consisting of nested tuples of distinct symbols:
\[
	S_\pm\subseteq\sL_\pm,\quad\text{and}\quad
	(\lambda,\mu,\lambda')\in\sL_\pm\text{ for } \lambda,\lambda'\in\sL_\pm,\ \mu\in\sL_\mp,
\]
if $\lambda,\mu,\lambda'$ consist of distinct symbols. For convenience, we set $\TP{\lambda}{\mu}{\lambda'}:=(\lambda,\mu,\lambda')$. Let $\sgn(\lambda)$ denote the \emph{sign} of the letter $\lambda$, i.e.\ $\lambda\in\sL_{\sgn(\lambda)}$. We note that $\sL$ is a finite set, e.g.\ $m=2$ yields
\[
	\sL_+ = \{x_1,x_2,\TP{x_1}{y_1}{x_2},\TP{x_1}{y_2}{x_1},\TP{x_2}{y_1}{x_1},\TP{x_2}{y_2}{x_1}\},
\]
and $\sL_-$ is obtain from $\sL_+$ by interchanging $x$- and $y$-symbols. For $m>2$, the brackets become important, e.g.
\[
	\TP{x_1}{\TP{y_1}{x_2}{y_2}}{x_3}\neq\TP{\TP{x_1}{y_1}{x_2}}{y_2}{x_3}.
\]
Omitting all brackets, letters are sequences of distinct symbols of alternating sign.

A \emph{word} is a finite sequence of letters, $w = \lambda_1\cdots\lambda_t$ with $\lambda_j\in\sL$. The number of letters in $w$ is called the \emph{length} of $w$ and denoted by $\ell(w)$. The main object of study is the set $\sW_m$ consisting of words in which each symbol occurs at most once, and for $1\leq j\leq m$, the symbol $x_j$ occurs if and only if $y_j$ occurs. Let $J_w$ denote the set of indices $j\in\{1,\ldots,m\}$ such that $x_j$ and $y_j$ occur in $w$.  We note that any word $w\in\sW_m$ has even length. Let $\ZZ\sW_m$ denote the free abelian group generated by $\sW_m$.

We now relate $\ZZ\sW_m$ to the universal enveloping algebra $\sU(\frakg)$ by the following procedure: Set
\[
	\TP{Z_1}{Z_2}{Z_3}:=[[Z_1,Z_2],Z_3]\quad\text{for}\quad
	Z_1,Z_2,Z_3\in\frakg,
\]
and associate to $w=w(x_1,\ldots,x_m,y_1,\ldots,y_m)\in\sW_m$ the element 
\[
	\Op(w) := \frac{1}{n^{m-|J_w|}}
		\sum_{\alpha_1,\ldots,\alpha_m=1}^n w(X_{\alpha_1},\ldots 
				X_{\alpha_m},Y_{\alpha_1},\ldots,Y_{\alpha_m})\in\sU(\frakg),
\]
i.e., replace in $w$ the symbols $x_j$ and $y_j$ by the basis elements $X_{\alpha_j}$ and $Y_{\alpha_j}$, and sum over all indices. Here, the factor $n^{|J_w|-m}$ compensates the overcounting in case that $w$ does not involve all symbols. 

It is straightforward to see that $\Op(w)$ is independent of the choice of the mutually dual bases $(X_\alpha)_\alpha$ of $\frakq_+$ and $(Y_{\alpha})_\alpha$ of $\frakq_-$, and it follows, that $\Op(\sigma)$ is in fact an $H$-invariant element. By $\ZZ$-linear extension, we thus obtain the $\ZZ$-module homomorphism
\[
	\Op\colon\ZZ\sW_m\to\sU(\frakg)^H,\ w\mapsto\Op(w),
\]
and according to \eqref{eq:firstisomorphism} we may interpret $\Op(w)$ as $G$-invariant differential operator on $X=G/H$. In particular, we note that
\[
	L_m = \Op(w_m)\quad\text{with}\quad w_m := y_1\cdots y_mx_1\cdots x_m.
\]
For any word $w\in\sW_m$, the order of the differential operator $\Op(w)$ is obviously bounded by the length $\ell(w)$ of the word.

\begin{remark}
	In the process of replacing the symbols $x_j$ and $y_j$ by the basis elements $X_{\alpha_j}$ and 
	$Y_{\alpha_j}$, respectively, the letters of a word are turned into commutators of elements in 
	$\frakq_+$ and $\frakq_-$. Recall that $\frakg=\frakq_-\oplus\frakh_\CC\oplus\frakq_+$ is a 
	grading. We thus obtain the commutation relations
	\begin{align}\label{eq:tripleproduct}
		\TP{\frakq_\pm}{\frakq_\mp}{\frakq_\pm}=[[\frakq_\pm,\frakq_\mp],\frakq_\pm]
		\subseteq\frakq_\pm,
	\end{align}
	i.e., a letter $\lambda$ is turned into an element in $\frakq_{\sgn(\lambda)}$. In fact, the 
	product \eqref{eq:tripleproduct} is the starting point for an algebraic description of Hermitian 
	symmetric spaces via so called Jordan theory. We refer to \cite{Lo77} and \cite{Sat80} for a 
	detailed description of this approach.
\end{remark}

\subsection{Relations and graphs}
We introduce relations on words that reflect some of the identities valid in the universal enveloping algebra. Let $R\subseteq\ZZ\sW_m$ be the submodule generated by the following relations, i.e., $R$ is the minimal submodule such that the following identities hold modulo $R$:
\begin{enumerate}[(R1)]
	\item \emph{Commutation relation I}\\
				Letters of the same sign commute, i.e.,
				\begin{align*}
					w\lambda\lambda'w' \equiv w\lambda'\lambda w'\mod R
				\end{align*}
				for all words $w,w'$ and letters $\lambda,\lambda'\in\sL$ with 
				$\sgn(\lambda)=\sgn(\lambda')$.
				\vspace{3mm}
	\item \emph{Commutation relation II}\\
				Letters of different sign satisfy
				\begin{align*}
					w\lambda\mu\lambda_1\cdots\lambda_t
						\equiv w\mu\lambda\lambda_1\cdots\lambda_t+
							\sum_{j=1}^t w\lambda_1\cdots[\lambda\mu\lambda_j]\cdots\lambda_t\mod R,
				\end{align*}
				for each word $w$ and letters $\lambda,\mu,\lambda_1,\ldots,\lambda_t\in\sL$
				with $\sgn(\lambda)\neq\sgn(\mu)$, where
				\begin{align*}
					[\lambda\mu\lambda_j] := \begin{cases}
						\TP{\lambda}{\mu}{\lambda_j} &\text{ if }\sgn(\lambda_j)=\sgn(\lambda),\\
						-\TP{\mu}{\lambda}{\lambda_j} &\text{ if }\sgn(\lambda_j)\neq\sgn(\lambda).
					\end{cases}
				\end{align*}
	\item \emph{Permutation symmetry}\\
				If $w\in\sW_m$ and $\tau\in\SymGp_m$ is a permutation of $\{1,\ldots,m\}$, then
				\begin{align*}
					w(x_1,&\ldots,x_m,y_1,\ldots,y_m)\\
						&\equiv w(x_{\tau(1)},\ldots,x_{\tau(m)},y_{\tau(1)},\ldots,y_{\tau(m)})\mod R.
				\end{align*}
\end{enumerate}

These relations are compatible with the structure of the universal enveloping algebra $\sU(\frakg)$.

\begin{lemma}\label{lem:relations}
	Let $w_1,w_2\in\ZZ\sW_m$. Then, $w_1\equiv w_2\mod R$ implies
	\[
		\Op(w_1)\equiv\Op(w_2)\mod\sU(\frakg)\frakh_\CC\cap\sU(\frakg)^H.
	\]
\end{lemma}
\begin{proof}
It suffices to check this statement for each of the relations (R1), (R2) and (R3). For the third relation (R3), this is obvious from the definition of $\Op$. Fix any set of elements $X_1,\ldots,X_m\in\frakq_+$ and $Y_1,\ldots,Y_m\in\frakq_-$ and for any word $w=w(x_1,\ldots,y_m)$, consider the element
\[
	\Op'(w):=w(X_1,\ldots,Y_m)\in\sU(\frakg).
\]
Since $\Op(w)$ is a linear combination of operators of type $\Op'$, it suffices to prove the statement for (R1) and (R2) with $\Op$ replaced by $\Op'$. We note that
\[
	\Op'(\lambda_1\cdots\lambda_t) =\Op'(\lambda_1)\cdots\Op'(\lambda_t).
\]
We first consider (R1), so let $\lambda,\lambda'$ be letters of equal sign. Note that a letter $\lambda$ is mapped by $\Op'$ onto an element in $\frakq_{\sgn(\lambda)}$. Since $\frakq_+$ and $\frakq_-$ are abelian, this implies that $\Op'(\lambda)$ and $\Op'(\lambda')$ commute in $\sU(\frakg)$. Now consider (R2), and let $\lambda$ and $\mu$ have different signs. Then,
\begin{align*}
	\Op'(w\lambda\mu\lambda_1\cdots\lambda_t)
	&= \Op'(w)\Op'(\lambda)\Op'(\mu)\Op'(\lambda_1)\cdots\Op'(\lambda_t)\\
	&=\Op'(w)\Op'(\mu)\Op'(\lambda)\Op'(\lambda_1)\cdots\Op'(\lambda_t)\\
	&\quad+\Op'(w)[\Op'(\lambda),\Op'(\mu)]\Op'(\lambda_1)\cdots\Op'(\lambda_t).
\end{align*}
Since $[\frakq_+,\frakq_-]\subseteq\frakh_\CC$, the commutator $[\Op'(\lambda),\Op'(\mu)]$ is an element of $\frakh_\CC$. Commuting this element successively to the right of the expression, yields the statement.
\end{proof}

We next associate to each word $w=\lambda_1\ldots\lambda_{\ell(w)}\in\sW_m$ a (pseudo-) graph $\Gamma_w=(\sV_w,\sE_w,\epsilon_w)$ with the following data: 
The set $\sV_w$ of vertices consists of the letters that constitute $w$, 
\[
	\sV_w = \{\lambda_1,\ldots,\lambda_{\ell(w)}\}.
\]
For the definition of edges, recall that $J_w\subseteq\{1,\ldots,m\}$ is the set of indices such that $w$ contains $x_j$ and $y_j$. Set $\sE_w:=J_w$, and let $j\in\sE_w$ represent the edge that connects the letter $\lambda_{s(j)}$ which contains $x_j$ with the letter $\lambda_{t(j)}$ which contains $y_j$. Formally, this is encoded in the map
\[
	\epsilon_w\colon\sE_w\to\set{\{\lambda,\lambda'\}}{\lambda,\lambda'\in\sV_w},\
	j\mapsto\{\lambda_{s(j)},\lambda_{t(j)}\}.
\]
The graph $\Gamma_w$ thus contains $1\leq |J_w|\leq m$ edges. As an example, the graphs associated to the words $w=y_1y_2x_1x_2$ and $w'=x_1\TP{y_1}{x_2}{y_2}$ are given by
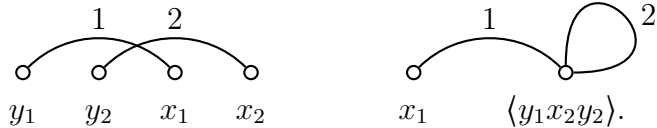
\tikzset{inner sep=0pt, vertex/.style={circle,draw,minimum size=5pt,thick}}
\begin{center}
	\begin{tikzpicture}[thick]
		\tikzstyle{every label}=[text height=1.2 em]
  	\clip (-0.5,-0.8) rectangle (3.5,1);
  	\node[vertex, label=below:$y_1$] (L1) at (0,0) {};
  	\node[vertex, label=below:$y_2$] (L2) at (1,0) {};
  	\node[vertex, label=below:$x_1$] (L3) at (2,0) {};
  	\node[vertex, label=below:$x_2$] (L4) at (3,0) {};
		\draw (L1) to [out=45,in=135]  node [above=3pt] {1} (L3);
		\draw (L2) to [out=45,in=135]  node [above=3pt] {2} (L4);
	\end{tikzpicture}\hspace{1cm}
	\begin{tikzpicture}[thick]
		\tikzstyle{every label}=[text height=1.2 em]
		\clip (-0.5,-0.8) rectangle (3.5,1);
  	\node[vertex, label=below:$x_1$] (L1) at (0,0) {};
  	\node[vertex, label=below:$\TP{y_1}{x_2}{y_2}.$] (L2) at (2,0) {};
		\draw (L1) to [out=45,in=135]  node [above=3pt] {1} (L2);
  	\draw (L2) .. controls +(0:2) and +(90:2) .. node [right=6pt] {2} (L2);
	\end{tikzpicture}
\end{center}
We call $w\in\sW_m$ \emph{irreducible}, if the graph $\Gamma_w$ is connected, and $w$ is called \emph{factorized}, if it is the product of irreducible words, i.e.,
\[
	w=w_1\cdots w_N\quad\text{ with irreducible $w_j\in\sW_m$}.
\]
Let $\sW_m^\fac\subseteq\sW_m$ denote the subset of factorized words in $\sW_m$. Not every word is factorized, as e.g.\ the word $w_m = y_1\cdots y_m x_1\cdots x_m$, associated to the higher Laplacian $L_m$, illustrates. If $w\in\sW_m$ is a product $w=w_1w_2$ of (not necessarily factorized) words $w_1,w_2\in\sW_m$, it is straightforward to check that
\[
	\Op(w) = \Op(w_1)\Op(w_2).
\]
We therefore aim to factorize words via the relations defined above.

\begin{proposition}\label{prop:factorization}
	Every word $w\in\sW_m$ is a $\ZZ$-linear combination of factorized words modulo $R$, i.e.\ 
	$\ZZ\sW_m = \ZZ\sW_m^\fac+R$.
\end{proposition}
\begin{proof}
We prove this by induction on the length of $w\in\sW_m$. Recall that the length of a word is always an even number. Any word $w$ of length two obviously is factorized, in fact it is irreducible, since letters have odd degree and hence there must be an edge connecting the first with the second letter. Now let $w=\lambda_1\cdots\lambda_L$ be of length $L=\ell(w)>2$, and let $\Gamma_w = \bigcup_{i=1}^N\Gamma_i$ be the decomposition of $\Gamma_w$ into connected components $\Gamma_i = (\sV_i,\sE_i,\epsilon_i)$. Then, $w$ is factorized if and only if the vertices $\sV_i$ in each component are adjacent letters in $w$, i.e.,\ there exist $0=t_0< t_1<t_2\cdots< t_N=L$ such that $\sV_i = \set{\lambda_j}{t_{i-1}\leq j\leq t_i}$ (up to a permutation of the connected components). Due to (R1) and (R2), 
\[
	w\equiv \lambda_{\tau(1)}\cdots \lambda_{\tau(L)} + \text{words of length $<L$}\mod R
\]
for all permutations $\tau\in\SymGp_L$ of $\{1,\ldots,L\}$. For suitable $\tau$, this transforms $w$ into a factorized word modulo words of shorter length. Applying the induction hypothesis, this proves the statement.
\end{proof}

\subsection{Main result}
We are now prepared to prove our main theorem. Recall the system $D_1,\ldots,D_r$ of algebraically independent generators for $\DD(G/H)$ defined in Section~\ref{subsec:generators}.

\begin{theorem}\label{thm:higherlaplacians}
	Let $X$ be a simple pseudo-Hermitian symmetric space of rank $r$ with irreducible isotropy 
	representation. For $m\in\NN$, the higher Laplacian $L_m$ is given by
	\begin{align*}
		L_{2k+1} &= c_k D_{k+1} + \text{polynomial in $D_1,\ldots, D_k$},
	\shortintertext{if $m=2k+1$ is odd, where $c_k$ is a non-zero constant, and}
		L_{2k} &= \text{polynomial in $D_1,\ldots, D_k$},
	\end{align*}
	if $m=2k$ is even.
\end{theorem}

Our main Theorem~\ref{thm:mainthm} is an immediate consequence of Theorem~\ref{thm:higherlaplacians}, since the relation between $L_1, L_3,\ldots, L_{2r-1}$ and $D_1,D_2,\ldots, D_r$ can easily be inverted. Therefore, $L_1,L_3,\ldots,L_{2r-1}$ form a set of algebraically independent generators for $\DD_G(X)$. For $m$ even, Theorem~\ref{thm:higherlaplacians} also shows that $L_m$ is a polynomial in $L_1,L_3,\ldots, L_{m-1}$.\\

The rest of this section is devoted to the proof of Theorem~\ref{thm:higherlaplacians}. For $m\geq 2r$, this statement is obvious, since $D_1,\ldots,D_r$ are generators for $\DD_G(X)$. So let $m<2r$. We first consider the operator $\Op(w)$ associated to an arbitrary word $w\in\sW_m$, and later specialize to the case $w_m:=y_1\ldots y_mx_1\cdots x_m$ with $\Op(w_m)=L_m$. Recall that there is a graph $\Gamma_w$ attached to each word $w\in\sW_m$, and that a graph is called a \emph{tree}, if it is connected and contains no cycles, so in particular no loops.

\begin{proposition}\label{prop:wordoppoly}
	Let $m\leq 2r$ and $w\in\sW_m^\fac$ be a factorized word. Then,
	\[
		\Op(w) = d_w\,D_{h+1} + \text{polynomial in $D_1,\ldots,D_h$, where }
		h:=\left\lfloor\tfrac{m}{2}\right\rfloor.
	\]
	Moreover, $d_w\neq 0$ if and only if $m=2h+1$ is odd and $\Gamma_w$ is a tree with $m$ edges.
	In this case, $d_w=c_0$ is the constant defined by \eqref{eq:structureconstant}.
\end{proposition}
\begin{proof}
Let $w=w_1\cdots w_t$ denote the decomposition of $w$ into irreducible words $w_i\in\sW_m$, and let $\Gamma_i=(\sV_i,\sE_i,\epsilon_i)$ be the graph associated to $w_i$, so $\Gamma_w=\bigcup_{i=1}^t\Gamma_i$. Recall that the number of vertices of the graph $\Gamma_i$ coincides with the length of $w_i$, hence the order of the differential operator $\Op(w_i)$ is bounded by $|\sV_i|$, which is an even number. Due to Lemma~\ref{lem:lowdegreeoperatros}, this implies that $\Op(w_i)$ is a polynomial in $D_1,\ldots, D_{h_i}$ with $h_i:=|\sV_i|/2$. By definition of irreducibility, $\Gamma_i$ is connected, and in this case elementary graph theory provides the estimate
\begin{align}\label{eq:graphestimate}
	|\sV_i|\leq|\sE_i|+1.
\end{align}
We thus obtain
\begin{align}\label{eq:inequalities}
	h_i\leq\frac{|\sE_i|+1}{2}\leq\frac{m+1}{2}\leq h+1.
\end{align}
Since $\Op(w)=\Op(w_1)\cdots\Op(w_t)$, it remains to investigate the case where equality holds in this chain of estimates for some $i$. The last inequality is strict if and only if $m$ is even. So we may assume that $m=2h+1$ is odd. The second estimate is an equality if and only if $|\sE_i|=m$. In this case, $w=w_i$, since each of the graphs $\Gamma_1,\ldots,\Gamma_t$ contains at least one edge, and $\sum_{j=1}^t|\sE_j|\leq m$. We thus may assume that $w$ is irreducible and the associated graph $\Gamma_w$ contains $m$ edges. It is a standard fact from graph theory, that \eqref{eq:graphestimate} is in fact an equality if and only if the graph is a tree. It therefore remains to show that 
\[
	\Op(w) = c_0\cdot D_{h+1} + \text{polynomial in $D_1,\ldots,D_h$}
\]
for any word $w\in\sW_m$ such that the associated graph $\Gamma_m$ is a tree with $m$ edges. This follows from the subsequent proposition, which computes the top degree term of $\gamma(\Op(w))$. It turns out that this term coincides with $c_0\cdot\gamma(D_{h+1})$, hence $\Op(w)-c_0\cdot D_{h+1}$ is a differential operator of order less or equal to $h$.
\end{proof}

\begin{proposition}\label{prop:topdegreeterm}
	Let $w\in\sW_m$ be irreducible and such that the associated graph $\Gamma_w$ is a tree. Then, 
	\[
		\gamma(\Op(w)) = c_0\cdot\gamma(D_{\ell(w)/2}) + \text{terms of lower degree,}
	\]
	where $\gamma$ is the Harish-Chandra isomorphism \eqref{eq:CHiso}, and $c_0$ is the constant 
	defined by \eqref{eq:structureconstant}.
\end{proposition}
\begin{proof}
For $X\in\sU(\frakg)$, let $X_\fraka$ denote the $S\fraka_\CC$-part of $X$ according to the projection \eqref{eq:aprojection}, and let $X_\fraka^\topterm$ denote the top degree term of $X_\fraka$. Then, $\gamma(X)^\topterm = X_\fraka^\topterm$, since the $\rho$-shift in \eqref{eq:secondisomorphism} does not affect the top term. We next determine a formula for $X_\fraka^\topterm$. Since the Iwasawa decomposition of $\frakg_\CC$ into the sum of $\fraka_\CC$ and $\frakn_\CC\oplus\frakh_\CC$ is orthogonal with respect to the Killing form, and hence also orthogonal with respect to \eqref{eq:metric}, it follows that
\[
	X_\fraka = \sum_{j=1}^r (X|A_j')\,A_j\quad\text{for}\quad X\in\frakg,
\]
where $(A_j)_{j=1,\ldots,r}$ is the basis of $\fraka_\CC$ determined by Lemma~\ref{lem:CartanBasis}, and $(A'_j)_{j=1,\ldots,r}$ is the dual basis with respect to \eqref{eq:metric}. Now consider $X=X_1\cdots X_r\in\sU(\frakg)$ with $X_j\in\frakg$. Decomposing each $X_j$ into $X_j=X_{j,\frakn}+X_{j,\fraka}+X_{j,\frakh}$ according to $\frakg =\frakn_\CC\oplus\fraka_\CC\oplus\frakh_\CC$, yields
\[
	X_\fraka = X_{1,\fraka}\cdots X_{r,\fraka} + \text{terms of lower degree.}
\]
If $X_{j,\fraka}$ is non-zero for all $j$, it follows that
\[
	X_\fraka^\topterm = X_{1,\fraka}\cdots X_{r,\fraka}.
\]
Identifying $S\fraka_\CC$ with $\Poly(\fraka_\CC)$ via \eqref{eq:metric}, this corresponds to the polynomial
\[
	X_\fraka^\topterm(\zeta) = \prod_{j=1}^r(X_j|\zeta).
\]
Applying this formula to the operator $\Op(w)$, we obtain
\begin{align}\label{eq:graphsymbolformula}
	\Op(w)_\fraka^\topterm(\zeta)
		= \sum_{\alpha_j:j\in\sE_w}\prod_{\lambda\in\sV_w}
			(\lambda(X_{\alpha_1},\ldots,X_{\alpha_m},Y_{\alpha_1},\ldots,Y_{\alpha_m})|\zeta).
\end{align}
Here, the sum involves all indices $\alpha_j$ that correspond to edges in $\sE_w$, and $\alpha_j$ takes the values $1,\ldots, n$. We thus avoid the factor in the original definition of $\Op(w)$ which corrected some overcounting.\\
We will need two steps to evaluate formula \eqref{eq:graphsymbolformula}. In the first one, we show that 
\begin{align}\label{eq:firststep}
	\Op(w)_\fraka^\topterm(\zeta) = (\mu(\zeta^+,\ldots,\zeta^+,\zeta^-\ldots,\zeta^-)|\zeta)
\end{align}
for some letter $\mu\in\sL$ depending on $\ell(w)-1$ of the symbols $x_1,\ldots,y_m$, and where $\zeta=\zeta^++\zeta^-$ according to the decomposition $\fraka_\CC\subseteq\frakq_+\oplus\frakq_-$. The second step is to prove that the right hand side of \eqref{eq:firststep} coincides with $\gamma(D_{\ell(w)/2})$.\\
\emph{First step.} In order to prove \eqref{eq:firststep}, we extend the set of symbols $\sS$ to
\[
	\sS':=\sS_+'\cup\sS_-'\quad\text{with}\quad
	\sS'_+:=\{x_1,\ldots,x_m,\zeta^+\},\quad
	\sS'_-:=\{y_1,\ldots,y_m,\zeta^-\},
\]
and consider letters based on $\sS'$ which are defined in the same way as before with the exception that the additional symbols $\zeta^+$ and $\zeta^-$ may occur arbitrarily often. In addition, we demand that any letter contains at least one of the former symbols $\sS$. For $\lambda\in\sL'$, let $d_\zeta(\lambda)$ denote the number of $\zeta^+$ and $\zeta^-$ contained in $\lambda$. The set $\sW'_m$ of words and the associated graphs are defined as before. In this setting, we claim: If $w\in\sW'_m$ is an irreducible word whose graph $\Gamma_w$ is a tree, then
\begin{align}\label{eq:modifiedfirststep}
	\sum_{\alpha_j:j\in\sE_w}\prod_{\lambda\in\sV_w}
			(\lambda(X_{\alpha_1},\ldots,Y_{\alpha_m})|\zeta)
	= (\mu(\zeta^+,\ldots,\zeta^-)|\zeta)
\end{align}
for some letter $\mu\in\sL'$ with $|\mu|=\ell(w)-1+\sum_{\lambda\in\sV_w}d_\zeta(\lambda)$. In particular, this proves \eqref{eq:firststep}.\\
Fix an edge $j\in\sE_w$, and let $\lambda$, $\lambda'$ be the letters connected by $j$. Since $\Gamma_w$ is a tree, $\lambda\neq\lambda'$, and we may assume that $\lambda$ contains $x_j$, and $\lambda'$ contains $y_j$. We aim to evaluate the sum over $\alpha_j$ in \eqref{eq:modifiedfirststep}. In order to simplify notation, we assume that $\lambda$ has negative sign, the argument for positive sign follows in the same way. Since $(\frakq_-|\frakq_-)=0$, this assumption yields
\[
	(\lambda(X_{\alpha_1},\ldots Y_{\alpha_m})|\zeta)
		=(\lambda(X_{\alpha_1},\ldots Y_{\alpha_m})|\zeta^+).
\]
Then, there exists a letter $\mu'\in\sL'_-$ not containing the symbols $x_j$ and $y_j$, such that
\[
	(\lambda(X_{\alpha_1},\ldots,Y_{\alpha_m})|\zeta_+) = 
		(\mu'(X_{\alpha_1}\ldots,Y_{\alpha_m})|X_{\alpha_j}),
\]
i.e., we may interchange $X_{\alpha_j}$ and $\zeta^+$ at the cost of replacing $\lambda$ by a new letter $\mu'$, which involves $\zeta^+$ instead of $x_j$. This can be proved in a straightforward way by means of the recursive definition of letters and using the associativity of \eqref{eq:metric}. We omit the details. Applying the standard identity $\sum_{\alpha_j}(X|X_{\alpha_j})Y_{\alpha_j}=X$, we conclude that
\[
	\sum_{\alpha_j}(\lambda(X_{\alpha_1},\ldots,Y_{\alpha_m})|\zeta)
		(\lambda'(X_{\alpha_1},\ldots,Y_{\alpha_m})|\zeta)
	= (\lambda''(X_{\alpha_1},\ldots,Y_{\alpha_m})|\zeta),
\]
where $\lambda''$ is obtained from $\lambda'$ by replacing $y_j$ by $\mu'$. If $|\sE_w|=1$, this already proves \eqref{eq:modifiedfirststep} with $\mu=\lambda''$. For $|\sE_w|>1$, consider the word $w'$ obtained from $w$ by omitting $\lambda$ and replacing $\lambda'$ by $\lambda''$. It follows that the graph $\Gamma_{w'}$ is obtained from $\Gamma_w$ by identifying the vertices $\lambda$ and $\lambda'$ along the edge $j$, and then omitting this edge. In particular, $\Gamma_{w'}$ is a tree with less edges than $\Gamma_w$. Moreover, by construction, the left hand side of \eqref{eq:modifiedfirststep} remains the same, when $\sE_w$ and $\sV_w$ are replaced by $\sE_{w'}$ and $\sV_{w'}$, respectively. Therefore, \eqref{eq:modifiedfirststep} follows by induction on the number of edges.\\
\emph{Second step.} Recall from Lemma~\ref{lem:CartanBasis} the decomposition of the basis elements $A_j$ into $A_j=E_j+F_j$ with $E_j\in\frakq_+$, $F_j\in\frakq_-$. It follows, that 
\begin{align}\label{eq:zetapm}
	\zeta^+ = \sum_{j=1}^r\zeta_j E_j,\quad \zeta^- =\sum_{j=1}^r\zeta_j F_j,
\end{align}
where $\zeta_j=(\zeta|A_j')$, i.e., $\zeta = \sum_{j=1}^r\zeta_j A_j$. Let $\mu\in\sL_+$ by any letter of positive sign depending on $d$ symbols. We claim, that
\begin{align}\label{eq:muformula}
	\mu(\zeta^+,\ldots,\zeta^-) = \sum_{j=1}^r\zeta_j^d\,E_j,
\end{align}
and the same holds for $\mu\in\sL_-$ with $E_j$ replaced by $F_j$.\\
The proof goes by induction on $d$. For $d=1$, we have $\mu=x_1$ or $\mu=y_1$, and \eqref{eq:muformula} reduces to \eqref{eq:zetapm}. Now let $d>1$. We may assume $\mu$ has positive sign, the argument for negative sign is the same. Then, $\mu=\TP{\mu_1}{\mu_2}{\mu_3}$ for some letters $\mu_1,\mu_3\in\sL_+$, $\mu_2\in\sL_-$ depending on $d_1,d_2,d_3>0$ symbols, respectively, with $d_1+d_2+d_3=d$. Applying the induction hypothesis yields
\begin{align*}
	\mu(\zeta^+,\ldots,\zeta^-)
	&= [[\mu_1(\zeta^+,\ldots,\zeta^-),\mu_2(\zeta^+,\ldots,\zeta^-)],
		\mu_3(\zeta^+,\ldots,\zeta^-)]\\
	&= \sum_{j,k,\ell=1}^r\zeta_j^{d_1}\zeta_k^{d_2}\zeta_\ell^{d_3}[[E_j,F_k],E_\ell]\\
	&= \sum_{j=1}^r\zeta_j^{d_1+d_2+d_3} E_j,
\end{align*}
since $[[E_j,F_k],E_\ell] = \delta_{jk}\delta_{k\ell}E_j$ due to the commutation relations given in Lemma~\ref{lem:CartanBasis}. This proves \eqref{eq:muformula}, and it easily follows that
\begin{align}\label{eq:zetaproduct}
	(\mu(\zeta^+,\ldots,\zeta^-)|\zeta)
		= c_0\cdot(\zeta_1^{d+1}+\cdots+\zeta_r^{d+1}),
\end{align}
where $c_0$ is the constant defined by \eqref{eq:structureconstant}. In summary, \eqref{eq:firststep} and \eqref{eq:zetaproduct} show that
\[
	\Op(w)_\fraka^\topterm(\zeta) = c_0\cdot(\zeta_1^{\ell(w)}+\cdots\zeta_r^{\ell(w)})
	= c_0\cdot\gamma(D_{\ell(w)/2}),
\]
which proves Proposition~\ref{prop:topdegreeterm}.
\end{proof}

We continue with the proof of Theorem~\ref{thm:higherlaplacians}, and now specialize to the word $w_m=y_1\cdots y_mx_1\cdots x_m$ with $\Op(w_m)=L_m$. Due to Proposition~\ref{prop:factorization}, $w_m$ admits a representative
\begin{align}\label{eq:factorizedwm}
	w_m \equiv \sum_{w\in\sW_m^\fac} n_w\,w\mod R
\end{align}
with $n_w\in\ZZ$. Applying Proposition~\ref{prop:wordoppoly} to each word $w\in\sW_m^\fac$ shows that $L_{2k}$ is indeed a polynomial in $D_1,\ldots,D_k$, and $L_{2k+1}$ is of the form
\[
	L_{2k+1} = c_k D_{k+1} + \text{polynomial in $D_1,\ldots,D_k$},
\]
where the constant $c_k$ is given by
\[
	c_k = c_0\cdot
		\sum_{\substack{w\in\sW_m^\fac:\ \Gamma_w\text{ tree}\\\text{with $m$ edges}}} n_w.
\]
It remains to show that this constant is non-zero. For this purpose, we describe a procedure which determines a factorized representative \eqref{eq:factorizedwm} for $w_m$. Since we are only interested in terms that contribute to the constant $c_k$, we may omit along the calculation all words, that are products of words, and words whose graph contain cycles. This is justified by the following corollary of Proposition~\ref{prop:wordoppoly}.

\begin{corollary}\label{cor:omittingwords}
	Let $m=2k+1$ be odd, and $w\in\sW_m$ be a word satisfying one of the following conditions:
	\begin{enumerate}[\quad(i)]
		\item $\Gamma_w$ contains a cycle,
		\item $\Gamma_w$ has less than $m$ edges,
		\item $w=w_1w_2$ for words $w_1,w_2\in\sW_m$.
	\end{enumerate}
	Then, $\Op(w)$ is a polynomial in $D_1,\ldots,D_k$. 
\end{corollary}
\begin{proof}
Due to the proof of Proposition~\ref{prop:factorization}, a factorized representative for $w$ is obtained by successively permuting the letters of $w$, i.e., applying (R1) and (R2). In this process, words are produced whose graphs have the same number of edges as $\Gamma_w$, and if $\Gamma_w$ contains a cycle, then all other graphs also contain this cycle. Therefore, if $w$ satisfies (i) or (ii), then each word in a factorized representative for $w$ satisfies (i) or (ii). Lemma~\ref{lem:relations} and Proposition~\ref{prop:wordoppoly} now imply that $\Op(w)$ is a polynomial in $D_1,\ldots,D_k$. If $w=w_1w_2$ is a product, we have $\Op(w)=\Op(w_1)\Op(w_2)$, and since $w_1$ and $w_2$ are words whose graphs have less than $m$ edges, it follows form (ii) that also $\Op(w)$ is a polynomial in $D_1,\ldots, D_k$.
\end{proof}

\noindent
\emph{Factorized representative of $w_m$ modulo $R'$.}
We now determine a factorized representative of $w_m$ modulo words that do not contribute to the constant $c_k$. Formally, we set $R':=R+\ZZ\sW_m'$ with
\[
	\sW_m':=\set{w\in\sW_m}
	{w=w_1w_2\text{ for }w_1,w_2\in\sW_m,\text{ or $\Gamma_w$ contains a cycle}},
\]
and write $w_1\equiv'w_2$ to indicate equality of $w_1$ and $w_2$ modulo $R'$. As a first step, we apply (R1), (R2), and obtain 
\begin{align*}
	w_m&\equiv' y_m\cdots y_1x_m\cdots x_1 \\
		&\equiv' y_m\cdots y_2x_my_1x_{m-1}\cdots x_1
			- \sum_{j=2}^{m-1}y_m\cdots y_2x_{m-1}\cdots\TP{x_m}{y_1}{x_j}\cdots x_1.
\end{align*}
Here, in contrast to (R2), we omit the last term involving $\TP{x_m}{y_1}{x_1}$, since this is a word containing a loop, and hence belongs to $R'$. Renaming the indices by means of (R3) and reordering the letters with (R1) shows that
\begin{align*}
	w_m\equiv' y_m\cdots y_2x_my_1x_{m-1}\cdots x_1
			- (m-2)\,y_m\cdots y_2x_m\cdots x_4\TP{x_3}{y_1}{x_2}x_1.
\end{align*}
Repeating this process successively, we commute $y_1$ in the first term also with $x_{m-1}$, $x_{m-2}$ etc., and obtain
\begin{align*}
	w_m\equiv'-\tbinom{m-1}{2}\,y_m\cdots y_2x_m\cdots x_4\TP{x_3}{y_1}{x_2}x_1.
\end{align*}
The same procedure can be used to commute $y_2$ to the right. The result is
\begin{align}\label{eq:calulation}\begin{aligned}
	w_m&\equiv'-\tbinom{m-1}{2}\,y_m\cdots y_3x_m\cdots x_4\TP{x_3}{y_1}{x_2}x_1y_2\\
	&\qquad+\tbinom{m-1}{2}\tbinom{m-3}{2}\,y_m\cdots y_3x_m\cdots
		 x_6\TP{x_5}{y_2}{x_4}\TP{x_3}{y_1}{x_2}x_1.\end{aligned}
\end{align}
We remark that words involving $\TP{x_j}{y_2}{\TP{x_3}{y_1}{x_2}}$ or $\TP{x_j}{y_2}{x_1}$ which appear modulo $R$ do not occur modulo $R'$, since there are loops or cycles contained in the associated graphs. The next step is to commute $y_3$ to the right. For the first word in \eqref{eq:calulation}, this yields
\begin{align}\begin{aligned}\label{eq:calculation2}
	y_m\cdots &y_3x_m\cdots x_4\TP{x_3}{y_1}{x_2}x_1y_2\\
		&\equiv' -\tbinom{m-3}{2}y_m\cdots y_4x_m\cdots x_6\TP{x_5}{y_3}{x_4}\TP{x_3}{y_1}{x_2}x_1y_2,
	\end{aligned}
\end{align}
and the second word transforms to
\begin{align*}
	y_m\cdots &y_3x_m\cdots x_6\TP{x_5}{y_2}{x_4}\TP{x_3}{y_1}{x_2}x_1\\
		&\equiv' y_m\cdots y_4x_m\cdots x_6\TP{x_5}{y_2}{x_4}\TP{x_3}{y_1}{x_2}x_1y_3\\
			&\quad-\tbinom{m-5}{2}y_m\cdots y_4x_m\cdots
			 x_8y_2\TP{x_7}{y_4}{x_6}\TP{x_5}{y_3}{x_4}\TP{x_3}{y_1}{x_2}x_1.
\end{align*}
All other words occurring modulo $R$ either have cycles contained in the associated graphs, or are products of words in $\sW_m$, and hence belong to $R'$. We thus obtain
\begin{align*}
	w_m&\equiv'\tbinom{m-1}{2}\tbinom{m-3}{2}\,
			y_m\cdots y_4x_m\cdots x_6\TP{x_5}{y_3}{x_4}\TP{x_3}{y_1}{x_2}x_1y_2\\
	&\quad+\tbinom{m-1}{2}\tbinom{m-3}{2}\,
			y_m\cdots y_4x_m\cdots x_6\TP{x_5}{y_2}{x_4}\TP{x_3}{y_1}{x_2}x_1y_3\\
	&\quad-\tbinom{m-1}{2}\tbinom{m-3}{2}\tbinom{m-5}{2}\,
			y_m\cdots y_4x_m\cdots x_8\TP{x_7}{y_4}{x_6}\TP{x_5}{y_3}{x_4}\TP{x_3}{y_1}{x_2}x_1y_2.
\end{align*}
The next step is to commute $y_4$ to the right. The pattern is the following: In each of the words we can decide either to put $y_4$ to the very right position or to pair $y_4$ with two unpaired $x$-symbols, which changes the sign, and reordering the $x$-symbols yields an additional binomial factor. Continuing this process, and omitting products of words (as e.g.\ in \eqref{eq:calculation2}, where we omit the word with $y_3$ put to the very right position), we eventually end up with (recall that $m=2k+1$)
\begin{align*}
	w_m&\equiv' (-1)^k\tbinom{m-1}{2}\tbinom{m-3}{2}\cdots\tbinom{2}{2}\\
		&\quad\cdot
		\sum_{\substack{I\subseteq\{2,\ldots,m\},\\I=\{i_1<\cdots<i_k\}}}
		\TP{x_m}{y_{i_k}}{x_{m-1}}\cdots\TP{x_5}{y_{i_1}}{x_4}\TP{x_3}{y_1}{x_2}x_1
		\prod_{i\notin I} y_i.
\end{align*}
The graph of each word in this sum is a tree with $m$ edges. The factor in front simplifies to $\frac{(m-1)!}{(-2)^k}$, and applying Proposition~\ref{prop:wordoppoly} finally yields that $L_m$ satisfies
\[
	L_{2k+1} = c_kD_{k+1}
	+ \text{polynomial in $D_1,\ldots,D_k$},\quad
	\text{where}\quad c_k:=c_0\,\tfrac{(2k)!}{(-2)^k}\,\tbinom{2k}{k}
\]
with $c_0$ defined by \eqref{eq:structureconstant}. This completes the proof of Theorem~\ref{thm:higherlaplacians}.\qed

\section{Higher Laplacians on semisimple pseudo-Hermitian symmetric spaces}\label{sec:nonIrreducible}

Let $X$ be a semisimple pseudo-Hermitian symmetric space. Then, $X$ decomposes into the product of simple pseudo-Hermitian symmetric spaces,
\begin{align}\label{eq:simpledecomposition}
	X=X_1\times\cdots\times X_N,
\end{align}
and the displacement group of $G$ is the product of the displacement groups $G_k$ of the $X_k$,
\[
	G=G_1\times\cdots\times G_N.
\]
This also coincides with the identity component of the automorphism group $\Aut(X)$ of $X$. Moreover, since the decomposition \eqref{eq:simpledecomposition} coincides with the deRham--Wu decomposition \cite{Wu64} of $X$, it follows from the essential uniqueness of this decomposition that the automorphism group of $X$ is generated by the automorphisms of the components $X_k$ and by permutations of mutually isomorphic components, so
\[
	\Aut(X)=\Aut^0(X)\rtimes\Pi,\quad \Aut^0(X):=\Aut(X_1)\times\cdots\times\Aut(X_N),
\]
where $\Pi$ denotes the group of permutations of mutually isomorphic components in \eqref{eq:simpledecomposition}.

Recall that the higher Laplicians $L_m$ are invariant for the full automorphism group $\Aut(X)$. Therefore, if the higher Laplacians generate the algebra of $G$-invariant differential operators, then any $G$-invariant operator is in fact $\Aut(X)$-invariant. Since there exist $G$-invariant operators acting non-trivially just on a single component of $X$, it follows that the higher Laplacians cannot generate $\DD_G(X)$ if $X$ contains two isomorphic components, see also Remark~\ref{rmk:reducible}. Therefore, it is more appropriate to consider a priori the algebra $\DD_{\Aut(X)}(X)$ of $\Aut(X)$-invariant differential operators, and ask whether the higher Laplacians generate this algebra.

Let $\fraka_k$ be a Cartan subspace for the component $X_k$ of $X$, and let $W_k$ denote the corresponding Weyl group. For the following, we assume that each component has irreducible isotropy representation. Then, our main result shows that $G_k$-invariant differential operators are automatically $\Aut(X_k)$-invariant, and the Harish-Chandra isomorphism yields
\[
	\DD_{\Aut(X_k)}(X_k)\cong\Poly(\fraka_{k,\CC})^{W_k}.
\]
Therefore, the Harish-Chandra isomorphism $\gamma$ identifies the algebra of $\Aut^0(X)$-invariant differential operators with the tensor product
\begin{align}\label{eq:polyalg}
	\Poly(\fraka_{1,\CC})^{W_1}\otimes\cdots\otimes\Poly(\fraka_{r,\CC})^{W_r}
	\cong\Poly(\fraka_\CC)^{W_1\times\cdots\times W_r},
\end{align}
where $\fraka=\fraka_1\oplus\cdots\oplus\fraka_r$ is a Cartan subspace for $X$. Moreover, since $\gamma$ is equivariant for the commutation of mutually isomorphic components, we conclude that 
\begin{align}\label{eq:semisimpleHC}
	\DD_{\Aut(X)}(X)\cong \Poly(\fraka_\CC)^{(W_1\times\cdots\times W_r)\rtimes\Pi}.
\end{align}
Formally, we may set $W:=(W_1\times\cdots\times W_r)\rtimes\Pi$ and call this the Weyl group associated to $\fraka_\CC$. We show that our main Theorem~\ref{thm:mainthm} also applies to the $N$-fold product of a rank-1 pseudo Hermitian symmetric space.

\begin{theorem}
	Let $X=X_1\times\cdots\times X_1$ be the $N$-fold product of a rank-1 pseudo-Hermitian symmetric 
	space $X_1$. Then, the higher Laplacians
	\[
		L_1,L_3,\ldots,L_{2N-1}
	\]
	form a set of algebraically independent generators for $\DD_G(X)$. Moreover, for $m$ even, $L_m$ 
	is a polynomial in $L_1,L_3,\ldots,L_{m-1}$.
\end{theorem}
\begin{proof}
A review of the proof of Theorem~\ref{thm:mainthm} shows that the assumption that $X$ is simple and has irreducible isotropy representation is only essential in order to show that the polynomials $p_1,\ldots, p_r$ defined in \eqref{eq:polynomials} generate the image of the Harish-Chandra isomorphism. Here, the Harish-Chandra isomorphism is replaced by \eqref{eq:semisimpleHC}. Applying Lemma~\ref{lem:CartanBasis} to each component of $X$, we obtain a basis $A_1,\ldots, A_N$ of $\fraka_\CC$, and since all components are isomorphic, the constant $c_0$ in Lemma~\ref{lem:CartanBasis} is the same for each basis vector. The present assumption on $X$ yields that $W=(W_1\times\cdots\times W_N)\rtimes\Pi$ acts by signed permutation on the basis elements $A_1,\ldots,A_N$, as in the case of a simple pseudo-Hermitian symmetric space of rank $N$ and with irreducible isotropy representation. Therefore, the polynomials $p_1,\ldots,p_r$ defined in \eqref{eq:polynomials} still generate the image of the Harish-Chandra isomorphism, and the assertion on the higher Laplacians follows by the same analysis as in the simple case.
\end{proof}

It would be interesting to clarify the relation between the higher Laplacians $L_m$ and the polynomials $p_m$ more thoroughly. One might conjecture, that the higher Laplacians generate the algebra of $\Aut(X)$-invariant differential operators on $X$ if and only if the polynomials $p_m$ generate the algebra $\Poly(\fraka_\CC)^W$ given by \eqref{eq:semisimpleHC}.

\providecommand{\bysame}{\leavevmode\hbox to3em{\hrulefill}\thinspace}
\providecommand{\MR}{\relax\ifhmode\unskip\space\fi MR }
\providecommand{\MRhref}[2]{%
  \href{http://www.ams.org/mathscinet-getitem?mr=#1}{#2}
}
\providecommand{\href}[2]{#2}

\end{document}